\documentclass{article}

\usepackage[english]{babel}
\usepackage{amsmath, amssymb, amsthm, amsfonts, bm}
\usepackage{enumerate}
\usepackage{graphicx, color}
\usepackage[colorlinks=true]{hyperref}
\usepackage{enumitem}

\numberwithin{equation}{section}
\newtheorem{theorem}{Theorem}[section]
\newtheorem{lemma}{Lemma}[section]
\theoremstyle{definition}

\theoremstyle{remark}

\newcommand{\Rnum}[1]{\uppercase\expandafter{\romannumeral #1\relax}}

\newcommand{\mr}[1]{\mathrm{#1}}
\newcommand{\mb}[1]{\mathbb{#1}}
\newcommand{\mc}[1]{\mathcal{#1}}

\DeclareMathOperator{\Av}{Av} % Average operator

\def\clap#1{\hbox to 0pt{\hss#1\hss}}

\title{Sharp upper bounds for integral quantities related to the Bellman function of three variables of the dyadic maximal operator}
\author{Eleftherios N. Nikolidakis}
\date{\today}

\begin{document}
\maketitle

\begin{abstract}
We obtain sharp upper bounds for integral quantities related to the Bellman function of three integral variables of the dyadic maximal operator.
\end{abstract}

\section{Introduction} \label{sec:0}
The dyadic maximal operator on $\mb R^n$ is a useful tool in analysis and is defined by
\begin{equation} \label{eq:0p1}
	\mc M_d\varphi(x) = \sup\left\{ \frac{1}{|S|} \int_S |\varphi(u)|\,\mr du: x\in S,\ S\subseteq \mb R^n\ \text{is a dyadic cube} \right\},
\end{equation}
for every $\varphi\in L^1_\text{loc}(\mb R^n)$, where $|\cdot|$ denotes the Lebesgue measure on $\mb R^n$, and the dyadic cubes are those formed by the grids $2^{-N}\mb Z^n$, for $N=0, 1, 2, \ldots$.\\
It is well known that it satisfies the following weak type (1,1) inequality
\begin{equation} \label{eq:0p2}
	\left|\left\{ x\in\mb R^n: \mc M_d\varphi(x) > \lambda \right\}\right| \leq \frac{1}{\lambda} \int_{\left\{\mc M_d\varphi > \lambda\right\}} |\varphi(u)|\,\mr du,
\end{equation}
for every $\varphi\in L^1(\mb R^n)$, and every $\lambda>0$,
from which it is easy to get the following  $L^p$-inequality
\begin{equation} \label{eq:0p3}
	\|\mc M_d\varphi\|_p \leq \frac{p}{p-1} \|\varphi\|_p,
\end{equation}
for every $p>1$, and every $\varphi\in L^p(\mb R^n)$.
It is easy to see that the weak type inequality \eqref{eq:0p2} is the best possible. For refinements of this inequality one can consult \cite{6}.

It has also been proved that \eqref{eq:0p3} is best possible (see \cite{1} and \cite{2} for general martingales and \cite{18} for dyadic ones).
An approach for studying the behaviour of this maximal operator in more depth is the introduction of the so-called Bellman functions which play the role of generalized norms of $\mc M_d$. Such functions related to the $L^p$-inequality \eqref{eq:0p3} have been precisely identified in \cite{4}, \cite{5} and \cite{11}. For the study of the Bellman functions of $\mc M_d$, we use the notation $\Av_E(\psi)=\frac{1}{|E|} \int_E \psi$, whenever $E$ is a Lebesgue measurable subset of $\mb R^n$ of positive measure and $\psi$ is a real valued measurable function defined on $E$. We fix a dyadic cube  $Q$ and define the localized maximal operator $\mc M'_d\varphi$ as in \eqref{eq:0p1} but with the dyadic cubes $S$ being assumed to be contained in $Q$. Then for every $p>1$ we let
\begin{equation} \label{eq:0p4}
	B_p(f,F)=\sup\left\{ \frac{1}{|Q|} \int_Q (\mc M'_d\varphi)^p: \Av_Q(\varphi)=f,\ \Av_Q(\varphi^p)=F \right\},
\end{equation}
where $\varphi$ is nonnegative in $L^p(Q)$ and the variables $f, F$ satisfy $0<f^p\leq F$. By a scaling argument it is easy to see that \eqref{eq:0p4} is independent of the choice of $Q$ (so we may choose
$Q$ to be the unit cube $[0,1]^n$).
In \cite{5}, the function \eqref{eq:0p4} has been precisely identified for the first time. The proof has been given in a much more general setting of tree-like structures on probability spaces.

More precisely we consider a non-atomic probability space $(X,\mu)$ and let $\mc T$ be a family of measurable subsets of $X$, that has a tree-like structure similar to the one in the dyadic case (the exact definition will be given in Section \ref{sec:2}).
Then we define the dyadic maximal operator associated to $\mc T$, by
\begin{equation} \label{eq:0p5}
	\mc M_{\mc T}\varphi(x) = \sup \left\{ \frac{1}{\mu(I)} \int_I |\varphi|\,\mr \; d\mu: x\in I\in \mc T \right\},
\end{equation}
for every $\varphi\in L^1(X,\mu)$, $x\in X$.

This operator is related to the theory of martingales and satisfies essentially the same inequalities as $\mc M_d$ does. Now we define the corresponding Bellman function of four variables of $\mc M_{\mc T}$, by
\begin{multline} \label{eq:0p6}
	B_p^{\mc T}(f,F,L,k) = \sup \left\{ \int_K \left[ \max(\mc M_{\mc T}\varphi, L)\right]^p\mr \; d\mu: \varphi\geq 0, \int_X\varphi\,\mr \; d\mu=f, \right. \\  \left. \int_X\varphi^p\,\mr \; d\mu = F,\ K\subseteq X\ \text{measurable with}\ \mu(K)=k\right\},
\end{multline}
the variables $f, F, L, k$ satisfying $0<f^p\leq F $, $L\geq f$, $k\in (0,1]$.
The exact evaluation of \eqref{eq:0p6} is given in \cite{5}, for the cases where $k=1$ or $L=f$. In the first case the author (in \cite{5}) precisely identifies the function $B_p^{\mc T}(f,F,L,1)$ by evaluating it in a first stage for the case where $L=f$. That is he precisely identifies $B_p^{\mc T}(f,F,f,1)$ (in fact $B_p^{\mc T}(f,F,f,1)=F \omega_p (\frac{f^p}{F})^p$, where                         $\omega_p: [0,1] \to [1,\frac{p}{p-1}]$ is the inverse function $H^{-1}_p$, of $H_p(z) = -(p-1)z^p + pz^{p-1}$). 

The proof of the above mentioned evaluation relies on a one-parameter integral inequality which is proved by arguments based on a linearization of the dyadic maximal operator. More precisely the author in \cite{5} proves that the inequality

\begin{equation}\label{eq:0p7}
	F\geq \frac{1}{(\beta+1)^{p-1}} f^p + \frac{(p-1)\beta}{(\beta+1)^p} \int_X (M_{\mathcal{T}}\varphi)^p \; d\mu,
\end{equation}
is true for every non-negative value of the parameter $\beta$ and sharp for one that depends on $f$, $F$ and $p$, namely for $\beta=\omega_p (\frac{f^p}{F})-1$. This gives as a consequence an upper bound for $B_p^{\mc T}(f,F,f,1)$, which after several technical considerations is proved to be best possible.Then  by using several calculus arguments the author in \cite{5} provides the evaluation of $B_p^{\mc T}(f,F,L,1)$ for every $L\geq f$. 

Now in \cite{11} the authors give a direct proof of the evaluation of $B_p^{\mc T}(f,F,L,1)$ by using alternative methods. Moreover in the second case, where $L=f$, the author (in \cite{5}) uses the evaluation of $B_p^{\mc T}(f,F,f,1)$ and provides the evaluation of the more general $B_p^{\mc T}(f,F,f,k)$, $k\in (0,1]$.

Our aim (in the future) is to use the results of this article in order to approach the following Bellman function problem (of three integral variables)

\begin{multline} \label{eq:0p8}
	B_{p,q}^{\mc T}(f,A,F) = \sup \left\{ \int_X \left(\mc M_{\mc T}\varphi\right)^p\mr \; d\mu: \varphi\geq 0, \int_X\varphi\,\mr \; d\mu=f, \right. \\  \left. \int_X\varphi^q\,\mr \; d\mu = A,\ \int_X\varphi^p\,\mr \; d\mu = F\right\},
\end{multline}
where $1<q<p$, and the variables $f,A,F$ lie in the domain of definition of the above problem.

In this article we prove that whenever $0<\frac{x^q}{\kappa^{q-1}}<y\leq x^{\frac{p-q}{p-1}}\cdot z^{\frac{q-1}{p-1}}\;\Leftrightarrow\; 0<s_1^{\frac{q-1}{p-1}}\leq s_2<1$, (where $s_1,s_2)$ are defined below), we find a constant $t=t(s_1,s_2)$ for which if $h:(0,\kappa]\longrightarrow\mb{R}^{+}$ satisfies  $\int_{0}^{\kappa}h=x$\,,\; $\int_{0}^{\kappa}h^q=y$ and $\int_{0}^{\kappa}h^p=z$ then

	$$\int_{0}^{\kappa}\bigg(\frac{1}{t}\int_{0}^{t}h\bigg)^p dt\leq t^p(s_1,s_2)\cdot\int_{0}^{\kappa}h^p,$$
where $t(s_1,s_2)=t$ is the greatest element of $\big[{1,t(0)}\big]$ for which $F_{s_1,s_2}(t)\leq 0$ and $F_{s_1,s_2}$ is defined in Section 4. Moreover for each such fixed $s_1,s_2$ 
\[t=t(s_1,s_2)=\min\Big\{{t(\beta)\;:\;\beta\in\big[0,\tfrac{1}{p-1}\big]}\Big\}\]
where $t(\beta)=t(\beta,s_1,s_2)$ is defined by equation (\ref{eq:3p4}). That is we find a constant $t=t(s_1,s_2)$ for which the above inequality  is satisfied for all $h:(0,\kappa]\longrightarrow\mb{R}^{+}$ as mentioned above. Note that $s_1,s_2$ depend by a certain way on $x,y,z$, namely $s_1=\frac{x^p}{\kappa^{p-1}z}$,  $s_2=\frac{x^q}{\kappa^{q-1}y}$ and $F_{s_1,s_2}(\cdot)$ is given in terms of $s_1,s_2$.

We need to mention that the extremizers for the standard Bellman function $B_p^{\mc T}(f,F,f,1)$ have been studied in \cite{7}, and in \cite{9} for the case $0<p<1$. Also in \cite{8} the extremal sequences of functions for the respective Hardy operator problem have been studied.  Additionally further study of the dyadic maximal operator can be seen in \cite{10,11} where symmetrization principles for this operator are presented, while other approaches for the determination of certain Bellman functions are given in \cite{13,14,15,16,17}. Moreover results related to applications of dyadic maximal operators can be seen in \cite{12}.

\bigskip

\section{Preliminaries} \label{Section 2}

The following theorem has been proven in \cite{3}

\begin{theorem}\label{thm:1p1}
Let $q\in(1,p)$, $f>0$ and $\varphi\in L^p(X,\mu)$, with $\int_{X}\varphi=f$\,. Then the inequality 
\begin{align*}
\int_{X}(\mc{M}_{\mc{T}}\varphi)^p\,d\mu&\leq\frac{p(\beta+1)^q}{G}\int_{X}(\mc{M}_{\mc{T}}\varphi)^{p-q}\varphi^q\,d\mu+\frac{(p-q)(\beta+1)}{G}\,f^p\,+\\
&\quad+\frac{p(q-1)\beta}{G}\,f^{p-q}\int_{X}(\mc{M}_{\mc{T}}\varphi)^q\,d\mu-\frac{p(\beta+1)^q}{G}\,f^{p-q}\int_{X}\varphi^q\,d\mu\,,
\end{align*}
\makebox[\linewidth][s]{is true for every value of $\beta>0$, where $G$ is defined by $G=G(p,q,\beta)=$}\\
$q(p-1)(\beta+1)-p(q-1)$. 
\end{theorem}
By using the proof of the result in \cite{11}, one can now see that the following inequality is an immediate consequence:

If $h:(0,1]\longrightarrow\mb{R}^{+}$ is non-increasing, $h\in L^{p}$ with $\int_{0}^{1}h=f$,\, $\int_{0}^{1}h^q=A$, then for every $\beta>0$ the following inequality is true:
\begin{align}\label{eq:1p1}
\int_{0}^{1}\bigg(\frac{1}{t}\int_{0}^{t}h\bigg)^p dt&\leq\frac{p(\beta+1)^q}{G}\int_{0}^{1}\bigg(\frac{1}{t}\int_{0}^{t}h\bigg)^{p-q}h^q(t)\,dt+\frac{(p-q)(\beta+1)}{G}f^p\,+\nonumber\\
&+\frac{p(q-1)\beta}{G}f^{p-q}\int_{0}^{1}\bigg(\frac{1}{t}\int_{0}^{t}h\bigg)^q dt-\frac{p(\beta+1)^q}{G}f^{p-q}\!\cdot\!A
\end{align}
where $G$ is as in the statement of Theorem (\ref{thm:1p1}).

By using (\ref{eq:1p1}) and a simple linear change of variables (in the variable $t$), one easily shows that if $\kappa\in(0,1]$ and  $h:(0,\kappa]\longrightarrow\mb{R}^{+}$ is $L^p$-integrable and non-increasing, then the following inequality is true:

\begin{align}\label{eq:1p2}
	\int_{0}^{\kappa}\bigg(\frac{1}{t}\int_{0}^{t}h\bigg)^p dt&\leq\frac{p(\beta+1)^q}{G}\int_{0}^{\kappa}\bigg(\frac{1}{t}\int_{0}^{t}h\bigg)^{p-q}h^q(t)\,dt\,+\nonumber\\
	&\quad+\frac{(p-q)(\beta+1)}{G}\cdot\frac{1}{\kappa^{p-1}}\bigg(\int_{0}^{\kappa}h\bigg)^p\,+\nonumber\\
	&\qquad+\frac{p(q-1)\beta}{G}\bigg(\frac{1}{\kappa}\int_{0}^{\kappa}h\bigg)^{p-q}\cdot\int_{0}^{\kappa}\bigg(\frac{1}{t}\int_{0}^{t}h\bigg)^q dt\,-\nonumber\\
	&\hspace{0.9cm}-\frac{p(\beta+1)^q}{G}\bigg(\frac{1}{\kappa}\int_{0}^{\kappa}h\bigg)^{p-q}\cdot \int_{0}^{\kappa}h^q\,.
\end{align}

Suppose now that $(f,A,F)$ are given such that there exists $\kappa_0\in(0,1]$ with  
\begin{align}\label{eq:1p3}
\omega_q\Big({\frac{f^q}{\kappa_0^{q-1} A}}\Big)=\omega_p\Big({\frac{f^p}{\kappa_0^{p-1} F}}\Big):=\varepsilon_0\,.
\end{align}
Define $g:(0,\kappa_0]\longrightarrow\mb{R}^{+}$ by $g(t)=\vartheta\,t^{-1+\frac{1}{\varepsilon_0}}$,\;$t\in(0,\kappa_0]$, where $\varepsilon_0$ is defined by (\ref{eq:1p3}) and $\vartheta$ is such that 
\begin{align}\label{eq:1p4}
\int_{0}^{\kappa_0}g=f\quad\Leftrightarrow\quad\vartheta=\frac{f}{\kappa_0^{1/\varepsilon_0}\varepsilon_0}\,.
\end{align}

We notice now the following 
\begin{enumerate}
\item[{A)}]
By using (\ref{eq:1p3}), (\ref{eq:1p4}) and the form that $g$ has, we easily see that the following relations are satisfied:
\[\int_{0}^{\kappa_0}g^q=A\,,\qquad \int_{0}^{\kappa_0}g^p=F\,.\]
Moreover, after some standard simple calculations, we see that, for every  $\kappa\in(0,\kappa_0]$, the following identity holds true:
\[\frac{\big({\textstyle\int_{0}^{\kappa}g}\big)^p}{\kappa^{p-1}\textstyle\int_{0}^{\kappa}g^p}=\frac{f^p}{\kappa_0^{p-1}F}=\frac{\big({\textstyle\int_{0}^{\kappa_0}g}\big)^p}{\kappa_0^{p-1}\textstyle\int_{0}^{\kappa_0}g^p}\,.\]
\item[{B)}]
Moreover for the function that is constructed in comment (A) above, the following equality is true:
\begin{align}\label{eq:1p5}
\frac{1}{t}\int_{0}^{t}g(u)\,du=\varepsilon_0\cdot g(t)\,,\quad \forall\,t\in(0,\kappa_0]\,.
\end{align}
By using (\ref{eq:1p5}), it is straight forward to see that we obtain equality in (\ref{eq:1p2}), if we replace $h$ by $g$ and $\kappa$ by $\kappa_0$.
\end{enumerate}

\bigskip

\section{A first approach} \label{sec:2}

As we have already noticed in in the previous section, if $h:(0,\kappa]\longrightarrow\mb{R}^{+}$ is $L^p$-integrable and non-increasing, where $0< \kappa\leq 1$, the following inequality is true:
\begin{align}\label{eq:2p1}
	\int_{0}^{\kappa}\bigg(\frac{1}{t}\int_{0}^{t}h\bigg)^p dt&\leq\frac{p(\beta+1)^q}{G}\int_{0}^{\kappa}\bigg(\frac{1}{t}\int_{0}^{t}h\bigg)^{p-q}h^q(t)\,dt\,+\nonumber\\
	&\quad+\frac{(p-q)(\beta+1)}{G}\cdot\frac{1}{\kappa^{p-1}}\bigg(\int_{0}^{\kappa}h\bigg)^p\,+\nonumber\\
	&\qquad+\frac{p(q-1)\beta}{G}\cdot\frac{1}{\kappa^{p-q}}\bigg(\int_{0}^{\kappa}h\bigg)^{p-q}\cdot\int_{0}^{\kappa}\bigg(\frac{1}{t}\int_{0}^{t}h\bigg)^q dt\,-\nonumber\\
	&\hspace{0.9cm}-\frac{p(\beta+1)^q}{G}\cdot\frac{1}{\kappa^{p-q}}\bigg(\int_{0}^{\kappa}h\bigg)^{p-q}\cdot \int_{0}^{\kappa}h^q\,,
\end{align}
where $G=G(p,q,\beta)=q(p-1)(\beta+1)-p(q-1)$.

It is now well known, that the inequalities 
\begin{align}\label{eq:2p2}
\int_{0}^{\kappa}\bigg(\frac{1}{t}\int_{0}^{t}h\bigg)^q dt&\leq\int_{0}^{\kappa}h^q\,\omega_q\bigg({\frac{\big({\textstyle\int_{0}^{\kappa}h}\big)^q}{\kappa^{q-1}\textstyle\int_{0}^{\kappa}h^q}}\bigg)^q\leq\nonumber\\
&\leq \frac{\beta+1}{\beta}\cdot\frac{(\beta+1)^{q-1}{\textstyle\int_{0}^{\kappa}h^q}-\frac{1}{\kappa^{q-1}}\big({\textstyle\int_{0}^{\kappa}h}\big)^q}{q-1}
\end{align}
are true, for every $\beta>0$, where the right-hand side of (\ref{eq:2p2}) attains its minimum value for $\beta=\omega_q\Big({\frac{\big({\textstyle\int_{0}^{\kappa}h}\big)^q}{\kappa^{q-1}\textstyle\int_{0}^{\kappa}h^q}}\Big)-1$ and this value is the middle term in the above chain of inequalities. We substitute the estimate described in  (\ref{eq:2p2}) in (\ref{eq:2p1}) and we get: 
\begin{small}
\begin{align}\label{eq:2p3}
	&\int_{0}^{\kappa}\bigg(\frac{1}{t}\int_{0}^{t}h\bigg)^p dt\leq\frac{p(\beta+1)^q}{G}\int_{0}^{\kappa}\bigg(\frac{1}{t}\int_{0}^{t}h\bigg)^{p-q}h^q(t)\,dt\,+\nonumber\\
	&\qquad+\frac{(p-q)(\beta+1)}{G}\cdot\frac{1}{\kappa^{p-1}}\bigg(\int_{0}^{\kappa}h\bigg)^p-\frac{p(\beta+1)^q}{G}\int_{0}^{\kappa}h^q\cdot\frac{1}{\kappa^{p-q}}\bigg(\int_{0}^{\kappa}h\bigg)^{p-q}\,+\nonumber\\
	&+\frac{p(q-1)\beta}{G}\,\frac{1}{\kappa^{p-q}}\bigg(\int_{0}^{\kappa}h\bigg)^{p-q}\frac{\beta+1}{\beta(q-1)}\bigg[{(\beta+1)^{q-1}\int_{0}^{\kappa}h^q-\frac{1}{\kappa^{q-1}}\bigg(\int_{0}^{\kappa}h\bigg)^{q}}\bigg]
\end{align}
\end{small}
\!\!\!We apply now H\"{o}lder's inequality on the first term on the right side of (\ref{eq:2p3}) and, denoting  $I:=\int_{0}^{\kappa}\big(\frac{1}{t}\int_{0}^{t}h\big)^p dt$, we get 
\begin{align}\label{eq:2p4}
&I\leq\frac{p(\beta+1)^q}{G}\,I^{\frac{p-q}{p}}\bigg({\int_{0}^{\kappa}h^p}\bigg)^{\frac{q}{p}}+\frac{(p-q)(\beta+1)}{G}\cdot\frac{1}{\kappa^{p-1}}\bigg(\int_{0}^{\kappa}h\bigg)^{p}\,+\nonumber\\
	&+\frac{p(q-1)\beta}{G}\,\frac{1}{\kappa^{p-q}}\bigg(\int_{0}^{\kappa}h\bigg)^{p-q}\frac{\beta+1}{\beta(q-1)}\bigg[{(\beta+1)^{q-1}\int_{0}^{\kappa}h^q-\frac{1}{\kappa^{q-1}}\bigg(\int_{0}^{\kappa}h\bigg)^{q}}\bigg]\,-\nonumber\\
	&\hspace{4.3cm}-\frac{p(\beta+1)^q}{G}\int_{0}^{\kappa}h^q\cdot\frac{1}{\kappa^{p-q}}\bigg(\int_{0}^{\kappa}h\bigg)^{p-q}\,.
\end{align}
We denote now $\int_{0}^{\kappa}h=x$\,,\; $\int_{0}^{\kappa}h^q=y$ and $\int_{0}^{\kappa}h^p=z$\,. Then (\ref{eq:2p4}) becomes:
\begin{align}\label{eq:2p5}
	I&\leq\frac{p(\beta+1)^q}{G}\,I^{\frac{p-q}{p}}\cdot z^{\frac{q}{p}}+\frac{(p-q)(\beta+1)}{G}\cdot\frac{1}{\kappa^{p-1}}\cdot x^{p}\,+\nonumber\\
	&\;\;+\frac{p(q-1)\beta}{G}\,\frac{1}{\kappa^{p-q}}\,x^{p-q}\frac{(\beta+1)^q}{\beta(q-1)}\,y -\frac{p(q-1)\beta}{G}\,\frac{1}{\kappa^{p-q}}\,x^{p-q}\frac{\beta+1}{\beta(q-1)}\,\frac{1}{\kappa^{q-1}}\,x^{q}\,-\nonumber\\
	&\hspace{5.5cm}-\frac{p(\beta+1)^q}{G}\,y\cdot\frac{1}{\kappa^{p-q}}\,x^{p-q}\,.
\end{align}
Then (\ref{eq:2p5}), gives after simple cancellations of terms, the following inequality
\begin{align}\label{eq:2p6}
I\leq\frac{p(\beta+1)^q}{G}\,I^{\frac{p-q}{p}}\cdot z^{\frac{q}{p}}-\frac{q}{G}\,(\beta+1)\,\frac{x^p}{\kappa^{p-1}}\,.
\end{align}
Set now $J:=\frac{\int_{0}^{\kappa}\big(\frac{1}{t}\int_{0}^{t}h\big)^p dt}{\int_{0}^{\kappa}h^p}=\frac{I}{z}$\,, and (\ref{eq:2p6}) gives:
\begin{align}\label{eq:2p7}
	J\leq\frac{p(\beta+1)^q}{G}\,J^{1-\frac{q}{p}}-\frac{q}{G}\,(\beta+1)\,\frac{x^p}{\kappa^{p-1}z}\,.
\end{align}
Denoting by $s_1=s_1(x,z)=\frac{x^p}{\kappa^{p-1}z}$\,, (\ref{eq:2p7}) becomes:
\begin{align}\label{eq:2p8}
	J^{1-\frac{q}{p}}-A_{\beta}J\geq\frac{q}{p}\,\frac{1}{(\beta+1)^{q-1}}\,s_1\,,  
\end{align}
where 
\begin{align}\label{eq:2p9}
A_{\beta}&:=\frac{G}{p(\beta+1)^{q}}=\frac{p(q-1)\beta+(p-q)(\beta+1)}{p(\beta+1)^{q}}=\nonumber\\
&=\frac{p-q}{p}\,\frac{1}{(\beta+1)^{q-1}}+\frac{(q-1)\beta}{(\beta+1)^{q}}\,.
\end{align}
Write $h_{\beta}(y)=y^{p-q}-A_{\beta}y^p$\,, for $y\geq1$, thus by setting $J'=J^{\frac{1}{p}}$, (\ref{eq:2p8}) gives 
\[h_{\beta}(J')\geq\frac{q}{p}\,\frac{1}{(\beta+1)^{q-1}}\,s_1\,.\]
But the function $h_{\beta}:[1,+\infty)\longrightarrow\mb{R}$ is strictly decreasing as we shall see now. Indeed
\begin{align}\label{eq:2p10}
h_{\beta}'(y)&=(p-q)y^{p-q-1}-A_{\beta}\cdot p y^{p-1}=\nonumber\\
 &=y^{p-1}\big({y^{-q}(p-q)-p\cdot A_{\beta}}\big)\cong y^{-q}(p-q)-p\cdot A_{\beta}\,,
\end{align}
where the symbol $\cong$ means that the respective quantities involved are of the same sign. Now (\ref{eq:2p10}), since $y\geq1$ gives
\begin{align}\label{eq:2p11}
h'_{\beta}(y)\underset{y\geq1}{\underset{\raisebox{1pt}{$\sim$}}{\leq}}(p-q)-p\cdot A_{\beta}\,,
\end{align}
with the obvious meaning for the symbol $\underset{\raisebox{1pt}{$\sim$}}{\leq}$\,.

We now prove that $A_{\beta}>\frac{p-q}{p},\; \forall\,\beta\in\big(0,\frac{1}{p-1}\big)$\,. We write  \[A_{\beta}=\frac{p-q}{p}\,\frac{1}{(\beta+1)^{q-1}}+\,\frac{(q-1)\beta}{(\beta+1)^{q}}\,,\] using (\ref{eq:2p9}) and then by setting $\mc{K}(\beta)$ the right side of the above equality we get  $\mc{K}(0)=\frac{p-q}{p}$, and $\mc{K}'(\beta)>0,\; \forall\,\beta\in\big(0,\frac{1}{p-1}\big)$\,. Indeed: 
\begin{align*}
\mc{K}'(\beta)&=\bigg({\frac{G(\beta)}{p(\beta+1)^{q}}}\bigg)'_{\beta}\cong q(p-1)(\beta+1)^{q}\cdot p-G\cdot p\cdot q\cdot(\beta+1)^{q-1}\cong\\
 &\cong (p-1)(\beta+1)-G=(p-1)(\beta+1)-q(p-1)(\beta+1)+p(q-1)=\\
 &=(\beta+1)\big({(p-1)-qp+q}\big)+p(q-1)=\\
 &=(\beta+1)\big({(p-1)-q(p-1)}\big)+p(q-1)=\\
&=-(p-1)(q-1)(\beta+1)+p(q-1)\stackrel{0\leq \beta<\frac{1}{p-1}}{>}\\
&>-(p-1)(q-1)\,\frac{p}{p-1}+p(q-1)=0\,.
\end{align*}
Then
\[\mc{K}(\beta)=A_{\beta}>\mc{K}(0)=\frac{p-q}{p}\,,\quad\forall\,\beta\in\big(0,\tfrac{1}{p-1}\big)\,.\]
Thus by (\ref{eq:2p11}) above, we set $h'_{\beta}(y)<0\,,\;\;\forall\,y\geq1$, thus $h_{\beta}:[1,+\infty)\longrightarrow\big({-\infty,h_{\beta}(1)}\big]$ is strictly decreasing.

Note that as we have seen above:
\[h_{\beta}(1)\geq h_{\beta}(J')\geq \frac{q}{p}\,\frac{1}{(\beta+1)^{q-1}}\,s_1\,,\]
(since obviously $J'\geq 1$), thus if we denote by $L:=\frac{q}{p}\,\frac{1}{(\beta+1)^{q-1}}\,s_1$ we have that:
\begin{align*}
h_{\beta}(J')\geq L\quad&\Rightarrow\quad J'\leq h^{-1}_{\beta}(L)\\
    &\Rightarrow \quad \int_{0}^{\kappa}\bigg(\frac{1}{t}\int_{0}^{t}h\bigg)^p dt\leq t_1^p\int_{0}^{\kappa}h^p\,,
\end{align*}
where  $t_1=t_1(\beta)$, is the value of $h^{-1}_{\beta}(L)$, that is 
\begin{align}\label{eq:2p12}
t_1=h^{-1}_{\beta}(L)\quad&\Leftrightarrow\quad h_{\beta}(t_1)=L\nonumber\\
&\Leftrightarrow\quad t_1^{p-q}-A_{\beta}t_1^p=L:=\frac{q}{p}\,\frac{1}{(\beta+1)^{q-1}}\,s_1\,.
\end{align}
Until now we have proved that for every $\beta\in\big(0,\tfrac{1}{p-1}\big)$ there exists $t_1(\beta)\geq1$ such that for every $h$ as above the following inequality is true
\begin{align}\label{eq:2p13}
\int_{0}^{\kappa}\bigg(\frac{1}{t}\int_{0}^{t}h\bigg)^p dt\leq t_1^p(\beta)\cdot\int_{0}^{\kappa}h^p\,,
\end{align}
where $t_1=t_1(\beta)$ is given by the equation (\ref{eq:2p12}). We wish now to find the minimum  of  $\big\{{t_1(\beta)\;:\;\beta\in\big(0,\tfrac{1}{p-1}\big)}\}$, in order to find the exact estimate that (\ref{eq:2p13}) gives for the integrals that are involved. \\
First of all if we set $\beta=\beta_0:=\omega_p(s_1)-1$, we notice that $t_1(\beta_0)=\omega_p(s_1)$. For the proof of this fact we move in the following manner: We prove the equality:
\begin{align}\label{eq:2p14}
(t')^{p-q}-\frac{G(\beta_0)}{p(\beta_0+1)^{q}}(t')^{p}=\frac{q}{p}\,\frac{1}{(\beta_0+1)^{q-1}}\,s_1
\end{align}
where $t'=\beta_0+1=\omega_p(s_1)$\,. Indeed (\ref{eq:2p14}) is equivalent to:
\begin{align*}
&(\beta_0+1)^{p-q}-\frac{G(\beta_0)}{p(\beta_0+1)^{q}}(\beta_0+1)^{p}=\frac{q}{p}\,\frac{1}{(\beta_0+1)^{q-1}}\,s_1&\Leftrightarrow\\
&-\frac{G(\beta_0)}{p}(\beta_0+1)^{p-q}+(\beta_0+1)^{p-q}=\frac{q}{p}\,\frac{1}{(\beta_0+1)^{q-1}}\,H_p(\beta_0+1)&\Leftrightarrow\\
&(\beta_0+1)^{p-q}(\beta_0+1)^{q-1}\bigg[{1-\frac{G(\beta_0)}{p}}\bigg]=\frac{q}{p}\,\Big({p(\beta_0+1)^{p-1}-(p-1)(\beta_0+1)^{p}}\Big) &\Leftrightarrow\\
&1-\frac{G(\beta_0)}{p}=q-\frac{q}{p}(p-1)(\beta_0+1)&\Leftrightarrow\\
&p-G(\beta_0)=pq-q(p-1)(\beta_0+1)&\Leftrightarrow\\
&G(\beta_0)=q(p-1)(\beta_0+1)-p(q-1)\,,
\end{align*}
which is true by the definition of $G(\beta_0)$\,.
Thus by (\ref{eq:2p14}) and the definition of $t_1(\beta_0)$ we get $t_1(\beta_0)=t'=\beta_0+1=\omega_p(s_1)$\,.\\
We now show that $\min\big\{{t_1(\beta)\;:\;\beta\in\big(0,\tfrac{1}{p-1}\big)}\big\}=t_1(\beta_0)=\omega_p(s_1)$ (obviously by definition $\beta_0\in\big(0,\tfrac{1}{p-1}\big)$)\,. \\
More precisely we prove that:
\begin{align}\label{eq:2p15}
t_1(\beta)>\omega_p(s_1)\,,\quad \forall\,\beta\in\big(0,\tfrac{1}{p-1}\big)\setminus\{{\beta_0}\}
\end{align}
Let us suppose that (\ref{eq:2p15}) is not true. Then there exists $\beta\in\big(0,\tfrac{1}{p-1}\big)\setminus\{{\beta_0}\}$ for which 
\begin{align}\label{eq:2p16}
t_1(\beta)&\leq\omega_p(s_1)\quad\Rightarrow\nonumber\\
t_1(\beta)^{p-q}-\frac{G}{p(\beta+1)^{q}}\,t_1(\beta)^{p}&\geq \big({\omega_p(s_1)}\big)^{p-q}-\frac{G}{p(\beta+1)^{q}}\,\big({\omega_p(s_1)}\big)^{p}\,,
\end{align}
because of the type of the monotonicity of  $h_{\beta}(y),\,y\geq1$. Then 
\[(\ref{eq:2p16})\quad\Rightarrow\quad \frac{q}{p}\,\frac{1}{(\beta+1)^{q-1}}\,s_1\geq \big({\omega_p(s_1)}\big)^{p-q}-\frac{G}{p(\beta+1)^{q}}\,\big({\omega_p(s_1)}\big)^{p}\,.\] 
At this point we prove the following:
\begin{lemma}\label{lem:2p1}
The inequality
\begin{align}\label{eq:2p17}
\frac{q}{p}\,\frac{s_1}{(\beta+1)^{q-1}}\leq \big({\omega_p(s_1)}\big)^{p-q}-\frac{G}{p(\beta+1)^{q}}\,\big({\omega_p(s_1)}\big)^{p}
\end{align}
is true, for every $\beta\in\big(0,\tfrac{1}{p-1}\big)$, with the only case of equality being: $\beta=\omega_p(s_1)-1=:\beta_0$\,.
\end{lemma}
\begin{proof}\leavevmode
\begin{align}\label{eq:2p18}
(\ref{eq:2p17})\quad\Leftrightarrow\quad\frac{\big({\omega_p(s_1)}\big)^{p-q}}{s_1}(\beta+1)^{q-1}-\frac{G(\beta)}{p\,s_1(\beta+1)}\,\big({\omega_p(s_1)}\big)^{p}\geq \frac{q}{p}\,.
\end{align}
Set $\omega_p(s_1)=\lambda\;\;\Leftrightarrow\;\; s_1=H_p(\lambda)$, where since $s_1\in(0,1]$, we get that $\lambda\in\big[{1,\frac{p}{p-1}}\big)$. Thus (\ref{eq:2p18}) becomes 
\begin{align}\label{eq:2p19}
\Delta(\lambda):=\frac{\lambda^{p-q}}{H_p(\lambda)}(\beta+1)^{q-1}-\frac{q(p-1)(\beta+1)-p(q-1)}{p(\beta+1)}\,\frac{\lambda^p}{H_p(\lambda)}\geq\frac{q}{p}\,.
\end{align}
We check first the case of equality: We just need to prove that: $\Delta(\beta+1)=\frac{q}{p}$\,. We calculate 
\begin{align*}
\Delta(\beta+1)&=\frac{(\beta+1)^{p-1}}{p(\beta+1)^{p-1}-(p-1)(\beta+1)^{p}}\,-\\
	&\qquad-\Big({\frac{q(p-1)}{p}-\frac{q-1}{\beta+1}}\Big)\frac{(\beta+1)^{p}}{p(\beta+1)^{p-1}-(p-1)(\beta+1)^{p}}=\\
	 &=\frac{1}{p-(p-1)(\beta+1)}\biggl\{{1-\Big({\frac{q(p-1)}{p}-\frac{q-1}{\beta+1}}\Big)(\beta+1)}\biggr\}=\\
	 &=\frac{1}{p-(p-1)(\beta+1)}\Bigl\{{q-\frac{q}{p}(p-1)(\beta+1)}\Bigr\}=\frac{q}{p}\,.
\end{align*}
We now prove that in (\ref{eq:2p19}), we have strict inequality when $\lambda\neq \beta+1$\,.Thus we have to prove that for every $\lambda\in\big[{1,\frac{p}{p-1}}\big)\setminus\{{\beta+1}\}$, the following chain of equivalent inequalities are true:
\begin{align*}
&\lambda^{p-q}(\beta+1)^{q-1}-\Big({\frac{q(p-1)}{p}-\frac{q-1}{\beta+1}}\Big)\lambda^{p}>\frac{q}{p}\big({p\lambda^{p-1}-(p-1)\lambda^p}\big)&\Leftrightarrow\\
&(\beta+1)^{q-1}-\lambda^{q}\Big({\frac{q(p-1)}{p}-\frac{q-1}{\beta+1}}\Big)>\frac{q}{p}\,\lambda^{q-1}\big({p-(p-1)\lambda}\big)&\Leftrightarrow\\
&(\beta+1)^{q-1}+\lambda^{q}\frac{q-1}{\beta+1}>q\,\lambda^{q-1}&\Leftrightarrow\\
&q\,\lambda^{q-1}-(q-1)\frac{\lambda^q}{\beta+1}<(\beta+1)^{q-1}&\Leftrightarrow\\
&q\,\Big({\frac{\lambda}{\beta+1}}\Big)^{q-1}-(q-1)\Big({\frac{\lambda}{\beta+1}}\Big)^{q}<1\,, \;\forall\,\lambda\in\big[{1,\tfrac{p}{p-1}}\big)\setminus\{{\beta+1}\}\,.
\end{align*}
But this last inequality is obviously true since the function $H_q(s)=q\,s^{q-1}-(q-1)s^q$ attains its maximum value on the interval $(0,+\infty)$, at the unique point $s_0=1$, and this value is $H_q(1)=1$\,.
\end{proof}

The results of this section are thus described as follows:\\
By using the inequalities (\ref{eq:2p1}) and (\ref{eq:2p2}) we reached to the bound:
\begin{align}\label{eq:2p20}
\int_{0}^{\kappa}\bigg(\frac{1}{t}\int_{0}^{t}h\bigg)^p dt\leq \omega_p\Big({\frac{\big({\textstyle\int_{0}^{\kappa}h}\big)^p}{\kappa^{p-1}\textstyle\int_{0}^{\kappa}h^p}}\Big)^{p}\cdot\int_{0}^{\kappa}h^p
\end{align}
which holds for every $h:(0,\kappa]\longrightarrow\mb{R}^{+}$ $L^p$-integrable and non-increasing, where $0<\kappa\leq1$\,. This bound is, of course, well known (see \cite{5}, \cite{11}), but the approach that we gave in this section, will give us a direction of the methods that we shall apply in order to find estimates as in (\ref{eq:2p20}), by considering also the variable $y=\int_{0}^{\kappa}h^q$, as given. This will be done in the next section.

\bigskip

\section{Second approach}\label{sec:3}

We now determine a bound for the quantity: $I=\int_{0}^{\kappa}\big(\frac{1}{t}\int_{0}^{t}h\big)^p dt$, where $h:(0,\kappa]\longrightarrow\mb{R}^{+}$ is $L^p$-integrable, non-increasing, with $\int_{0}^{\kappa}h=x$\,,\; $\int_{0}^{\kappa}h^q=y$ and $\int_{0}^{\kappa}h^p=z$, for $1<q<p$, without ignoring the variable $y$. We move in the following manner: Remember that the following inequality is true (see the Preliminaries):
\begin{align*}
I=	\int_{0}^{\kappa}\bigg(\frac{1}{t}\int_{0}^{t}h\bigg)^p dt&\leq\frac{p(\beta+1)^q}{G}\int_{0}^{\kappa}\bigg(\frac{1}{t}\int_{0}^{t}h\bigg)^{p-q}h^q(t)\,dt\,+\nonumber\\
	&\quad+\frac{(p-q)(\beta+1)}{G}\cdot\frac{1}{\kappa^{p-1}}\bigg(\int_{0}^{\kappa}h\bigg)^p\,+\nonumber\\
	&\qquad+\frac{p(q-1)\beta}{G}\cdot\frac{1}{\kappa^{p-q}}\bigg(\int_{0}^{\kappa}h\bigg)^{p-q}\cdot\int_{0}^{\kappa}\bigg(\frac{1}{t}\int_{0}^{t}h\bigg)^q dt\,-\nonumber\\
	&\hspace{0.9cm}-\frac{p(\beta+1)^q}{G}\cdot\frac{1}{\kappa^{p-q}}\bigg(\int_{0}^{\kappa}h\bigg)^{p-q}\cdot \int_{0}^{\kappa}h^q\qquad\Rightarrow
\end{align*}
(by applying H\"{o}lder's inequality on the first term on the right)
\begin{align}\label{eq:3p1}
I&\leq\frac{p(\beta+1)^q}{G}\,I^{\frac{p-q}{p}}\cdot z^{\frac{q}{p}}+\frac{(p-q)}{G}(\beta+1)\cdot\frac{1}{\kappa^{p-1}}\cdot x^{p}\,-\nonumber\\
&\quad-\frac{p(\beta+1)^q}{G}\cdot\frac{1}{\kappa^{p-q}}x^{p-q}y     +\frac{p(q-1)\beta}{G}\,\frac{1}{\kappa^{p-q}}\,x^{p-q}y\cdot\omega_q\Big({\frac{x^q}{\kappa^{q-1}y}}\Big)^q\,,
\end{align}
where we have used the first of the inequalities (\ref{eq:2p2}).\\
Set $\delta=\Big({\frac{\int_{0}^{\kappa}\big(\frac{1}{t}\int_{0}^{t}h\big)^p dt}{\int_{0}^{\kappa}h^p}}\Big)^{1/p}=\big({\frac{I}{z}}\big)^{1/p}$\,. Thus we get:
\begin{align}\label{eq:3p2}
&\delta^p\leq \frac{p(\beta+1)^q}{G}\,\delta^{p-q}\,+\nonumber\\
&+\Bigg[{\frac{(p-q)}{G}(\beta+1)\,\frac{x^p}{\kappa^{p-1}z}-\frac{p(\beta+1)^q}{G}\,\frac{\frac{x^p}{\kappa^{p-1}z}}{\frac{x^q}{\kappa^{q-1}y}}+\frac{p(q-1)\beta}{G}\,\frac{\frac{x^p}{\kappa^{p-1}z}}{\frac{x^q}{\kappa^{q-1}y}}\cdot\omega_q\Big({\frac{x^q}{\kappa^{q-1}y}}\Big)^q}\Bigg]\nonumber\\
&=\frac{p(\beta+1)^q}{G}\,\delta^{p-q}+T_1(\beta)\,,
\end{align}
where $T_1(\beta)$ is the quantity in brackets in (\ref{eq:3p2}). Denote also $s_1=\frac{x^p}{\kappa^{p-1}z}$ and $s_2=\frac{x^q}{\kappa^{q-1}y}$\,. Note that, since  $\int_{0}^{\kappa}h=x$\,,\; $\int_{0}^{\kappa}h^q=y$\,, \; $\int_{0}^{\kappa}h^p=z  \Rightarrow \frac{x^q}{\kappa^{q-1}}\leq y\leq x^{\frac{p-q}{p-1}}\cdot  z^{\frac{q-a}{p-1}}$\,.\\
By the second inequality above we immediately get  $s_1^{q-1}\leq s_2^{p-1}$. Further $0<s_1$ and $s_2\leq 1$.\\
Define also (as we have done in Section \ref{sec:2}) : $h_{\beta}(y)=y^{p-q}-A_{\beta}y^p$\,,\; $y\geq1$, where $A_{\beta}:=\frac{G}{p(\beta+1)^q}$\,. Thus (\ref{eq:3p2}) describes the following inequality:
\begin{align}\label{eq:3p3}
h_{\beta}(\delta)\geq -\frac{G}{p(\beta+1)^q}\big({T_1(\beta)}\big)\,.
\end{align}
As it is proved in Section \ref{sec:2}, the function $h_{\beta}$ is strictly decreasing on $[1,+\infty)$, and $A_{\beta}$ satisfies:
\[A_{\beta}>\frac{p-q}{p}\,,\quad\forall\,\beta\in\big({0,\tfrac{1}{p-1}}\big)\,.\]
Since $\delta\geq 1$ and thus $h_{\beta}(\delta)\leq h_{\beta}(1)$, and also $\lim_{y\to +\infty}h_{\beta}(y)=-\infty$, we have that for every $\beta \in\big({0,\tfrac{1}{p-1}}\big)$ there exists unique $t(\beta)=t(\beta,s_1,s_2)\geq1$, for which:
\begin{align}\label{eq:3p4}
h_{\beta}(t(\beta))&=-\frac{G}{p(\beta+1)^q}\cdot T_1(\beta)\quad&\Leftrightarrow\nonumber\\
\frac{G}{p(\beta+1)^q}\,t^p(\beta)-t^{p-q}(\beta)&=\frac{G}{p(\beta+1)^q}\cdot T_1(\beta)\quad&\Leftrightarrow\nonumber\\
t^p(\beta)-\frac{p(\beta+1)^q}{G}\,t^{p-q}(\beta)&=T_1(\beta)=\nonumber\\
=\frac{p-q}{G}(\beta+1)\,s_1+\,&\frac{s_1}{s_2}\Big({\frac{p(q-1)\beta}{G}\,\omega_q(s_2)^q-\frac{p(\beta+1)^q}{G}}\Big)
\end{align}
Note that for this choice of $t(\beta)$ we obviously have $\delta\leq t(\beta)$, while all the inequalities stated above, remain to hold, even in the case $\beta=0$, (by letting $\beta\to 0^{+}$ in (\ref{eq:3p1}) ). Moreover if we consider $t=\min_{0\leq\beta<\frac{1}{p-1}}t(\beta)$ we should have that $t\leq \omega_p(s_1)$ since we have used the strong part of inequality (\ref{eq:2p2}) in order to obtain the bound $t(\beta)$, and all the remaining arguments are exactly the same, as in Section \ref{sec:2}, where we have obtain the bound $t'_1=\min_{0<\beta<\frac{1}{p-1}}t_1(\beta)=\omega_p(s_1)$. We remind also that the quantity $G=G(p,q,\beta)$ is given by $G=q(p-1)(\beta+1)-p(q-1)\,,\; \forall\,\beta$. 

Our aim is to differentiate (\ref{eq:3p4}) in terms of $\beta$, so for this reason we calculate the following derivatives (with respect to $\beta)$ :
\begin{enumerate}[label=\roman*)]
\item 
$\Big({\dfrac{\beta+1}{G}}\Big)'=\dfrac{G-(\beta+1)q(p-1)}{G^2}=\dfrac{-p(q-1)}{G^2}$\,.
\item 
$\Big({\dfrac{\beta}{G}}\Big)'=\Big({\dfrac{\beta+1}{G}}\Big)'-\Big({\dfrac{1}{G}}\Big)'=\dfrac{-p(q-1)}{G^2}+\dfrac{G'}{G^2}=$\\
${}\hspace{0.9cm}=\dfrac{-p(q-1)+q(p-1)}{G^2}=\dfrac{p-q}{G^2}$\,,\;\;  and\vspace{0.3cm}
\item \leavevmode\vspace{-0.95cm}
\begin{align*}
\Big({\dfrac{(\beta+1)^{q}}{G}}\Big)'&=\dfrac{q(\beta+1)^{q-1}G-(\beta+1)^{q}q(p-1)}{G^2}=\\
&=\dfrac{q(\beta+1)^{q-1}}{G^2}\big({G-(p-1)(\beta+1)}\big)=\\
&=\dfrac{q(\beta+1)^{q-1}}{G^2}\big({q(p-1)(\beta+1)-p(q-1)-(p-1)(\beta+1)}\big)=\\
&=\dfrac{q(\beta+1)^{q-1}}{G^2}\big({(\beta+1)(p-1)(q-1)-p(q-1)}\big)=\\
&=\dfrac{q(q-1)(\beta+1)^{q-1}}{G^2}\big({(\beta+1)(p-1)-p}\big)\,.
\end{align*}
\end{enumerate}
We restrict ourselves, in the following range of values for $\beta$ : $\beta\in\big[{0,\frac{1}{p-1}}\big)$\,. We now proceed to differentiate  (\ref{eq:3p4}), with respect to $\beta$ in the range  $\beta\in\big({0,\frac{1}{p-1}}\big)$\,.\\
The relations that will come after, remain to hold even for $\beta=0$ or $\beta=\frac{1}{p-1}$, by defining $t'(0)=\lim_{\beta\to 0^{+}}t'(\beta)$ and  $t'\big({\frac{1}{p-1}}\big)=\lim_{\beta\to {\frac{1}{p-1}}^{-}}t'(\beta)$\,.\\
We get by (\ref{eq:3p4}) that:
\begin{align}\label{eq:3p5}
&p\,t^{p-1}(\beta)\, t'(\beta)-p(p-q)\,t^{p-q-1}(\beta)\,\frac{(\beta+1)^q}{G}\, t'(\beta)-p\,t^{p-q}(\beta)\,q(q-1)\,\frac{\vartheta(\beta)}{G^2}=\nonumber\\
&=(p-q)s_1\frac{\big({-p(q-1)}\big)}{G^2}+\frac{s_1}{s_2}\bigg[{p(q-1)\frac{p-q}{G^2}\,\omega_q(s_2)^q -\frac{p\,q(q-1)\vartheta(\beta)}{G^2}}\bigg],
\end{align}
where $\vartheta(\beta)=(p-1)(\beta+1)^q-p\,(\beta+1)^{q-1}$\,.\\
We consider the following case: There exists $\beta\in\big({0,\frac{1}{p-1}}\big)$ for which $t'(\beta)=0$\,. Then (\ref{eq:3p5}) implies 
\begin{align}\label{eq:3p6}
-q\,t^{p-q}(\beta)\,\vartheta(\beta)&=-(p-q)s_1+\frac{s_1}{s_2}\Big[{(p-q)\,\omega_q(s_2)^q -q\,\vartheta(\beta)}\Big]&\Rightarrow\nonumber\\
(p-q)s_1\Big({\frac{\omega_q(s_2)^q}{s_2}-1}\Big)&=q\,\vartheta(\beta)\Big({\frac{s_1}{s_2}-t^{p-q}(\beta)}\Big)&\Rightarrow\nonumber\\
t^{p-q}(\beta)&=\frac{s_1}{s_2}-(p-q)\,\frac{s_1\cdot\alpha(s_2)}{q\,\vartheta(\beta)}\,,
\end{align}
where $\alpha(s_2)$ is defined by:
\[\alpha(s_2)=\frac{\omega_q(s_2)^q}{s_2}-1\,,\quad s_2\in (0,1]\,.\]
Also, since  $\vartheta(\beta)=(p-1)(\beta+1)^q-p\,(\beta+1)^{q-1}$ we have
\begin{align*}
\vartheta'(\beta)&=q\,(p-1)(\beta+1)^{q-1}-p(q-1)\,(\beta+1)^{q-2}\cong\\
&\cong q(p-1)(\beta+1)-p(q-1)=\\
&=G>q(p-1)-p(q-1)=p-q>0\,.
\end{align*}
Thus $\vartheta:\big({0,\frac{1}{p-1}}\big)\longrightarrow\mb{R}^{-}$ is increasing with $\vartheta(0^{+})=-1$\,,\;  $\vartheta\big({{\frac{1}{p-1}^{-}}}\big)=0$\,.

\medskip

\noindent{\bfseries{Claim 1.}} $t=t(\beta)$ satisfies:
\begin{align}\label{eq:3p7}
t^p-\dfrac{p(\beta+1)^{q}}{G}\,t^{p-q}=s_1-\dfrac{p(\beta+1)^{q}}{G}\,\frac{s_1}{s_2}+p(q-1)\,\frac{\beta}{G}\,s_1\,\alpha(s_2)\,.
\end{align}
Indeed from (\ref{eq:3p4}) it is enough to prove the following:
\begin{align*}
&\frac{p-q}{G}(\beta+1)s_1+\frac{s_1}{s_2}\bigg({\frac{p(q-1)\beta}{G}\,\omega_q(s_2)^q-\frac{p(\beta+1)^q}{G}}\bigg)=\\
&{}\hspace{2.0cm}=s_1-\frac{p(\beta+1)^q}{G}\,\frac{s_1}{s_2}+p(q-1)\frac{\beta}{G}\,s_1\bigg({\frac{\omega_q(s_2)^q}{s_2}-1}\bigg)&\Leftrightarrow\\
&\frac{p-q}{G}(\beta+1)s_1+\frac{s_1}{s_2}\,\frac{p(q-1)\beta}{G}\,\omega_q(s_2)^q=\\
&{}\hspace{2.0cm}=s_1+p(q-1)\frac{\beta}{G}\,\frac{s_1}{s_2}\,\omega_q(s_2)^q-p(q-1)\frac{\beta}{G}\,s_1&\Leftrightarrow\\
&(p-q)(\beta+1)s_1=s_1\cdot G-p(q-1)\beta\,s_1 &\Leftrightarrow\\
&(p-q)(\beta+1)=q(p-1)(\beta+1)-p(q-1)-p(q-1)\beta&\Leftrightarrow\\
&(p-q)(\beta+1)+p(q-1)(\beta+1)-p(q-1)=q(p-1)(\beta+1)-p(q-1)&\Leftrightarrow\\
&(\beta+1)(pq-p+p-q)=q(p-1)(\beta+1)\,,
\end{align*}
which is obviously true. Thus (\ref{eq:3p7}) is also true.\\
Now by (\ref{eq:3p6}) and (\ref{eq:3p7}) we deduce
\begin{align}\label{eq:3p8}
&t^p=\frac{p(\beta+1)^q}{G}\bigg({\frac{s_1}{s_2}-\frac{p-q}{q}\,\frac{s_1\,\alpha(s_2)}{\vartheta(\beta)}}\bigg)+s_1-\frac{p(\beta+1)^q}{G}\,\frac{s_1}{s_2}\,+\nonumber\\
&{}\hspace{8.4cm}+p(q-1)\,\frac{\beta}{G}\,s_1\,\alpha(s_2)=\nonumber\\
&=s_1+p(q-1)\,\frac{\beta}{G}\,s_1\,\alpha(s_2)-\frac{p(\beta+1)^q}{G}\,\frac{p-q}{q}\,\frac{s_1\,\alpha(s_2)}{\vartheta(\beta)}=\nonumber\\
&=s_1+p(q-1)\,\frac{\beta}{G}\,s_1\,\alpha(s_2)-\frac{p(\beta+1)^q}{G}\,\frac{p-q}{q}\,\frac{s_1\,\alpha(s_2)}{(p-1)(\beta+1)^q-p\,(\beta+1)^{q-1}}=\nonumber\\
&=s_1+p(q-1)\,\frac{\beta}{G}\,s_1\cdot\alpha(s_2)-\frac{p(\beta+1)}{G}\,\frac{p-q}{q}\,\frac{s_1\,\alpha(s_2)}{(p-1)(\beta+1)-p}\qquad\Rightarrow\nonumber\\
&t^p=s_1+p\,\frac{s_1}{G}\bigg[{(q-1)\beta\,\alpha(s_2)-\frac{p-q}{q}\,\frac{(\beta+1)\,\alpha(s_2)}{(p-1)(\beta+1)-p}}\bigg]\qquad\Rightarrow\nonumber\\
&t^p=s_1+\frac{p\,s_1}{G}\,\alpha(s_2)\bigg[{(q-1)\beta-\frac{p-q}{q}\,\frac{(\beta+1)}{(p-1)(\beta+1)-p}}\bigg]=\nonumber\\
&\;\;\;=s_1+\frac{p\,s_1\,\alpha(s_2)}{G}\,\bigg[{(q-1)\beta-\frac{(p-q)(\beta+1)}{q(p-1)(\beta+1)-p(q-1)-p}}\bigg]\qquad\Rightarrow\nonumber\\
&t^p=s_1+\frac{p\,s_1\,\alpha(s_2)}{G}\,\bigg[{(q-1)\beta-\frac{(p-q)(\beta+1)}{G-p}}\bigg]\,.
\end{align}
Note that $(G-p)^{-1}$ is well defined since $\beta\in\big({0,\frac{1}{p-1}}\big)$. Moreover
\begin{align*}
q\vartheta(\beta)&=q(\beta+1)^{q-1}\big({(p-1)(\beta+1)-p}\big)=\\
 &=(\beta+1)^{q-1}\big({q(p-1)(\beta+1)-p(q-1)-p}\big)=\\
 &=(\beta+1)^{q-1}(G-p)\,.
\end{align*}
Thus (\ref{eq:3p6}) becomes:
\begin{align}\label{eq:3p9}
t^{p-q}=\frac{s_1}{s_2}-(p-q)\,\frac{s_1\,\alpha(s_2)}{(\beta+1)^{q-1}(G-p)}\,.
\end{align}
\makebox[\linewidth][s]{Then we simplify the form that (\ref{eq:3p8}) has, by considering  the term}\\ $\Big[{(q-1)\beta-\frac{(p-q)(\beta+1)}{G-p}}\Big]\Big/G$, which equals to:
\begin{align}\label{eq:3p10}
\frac{\Big[{(q-1)\beta-\frac{(p-q)(\beta+1)}{\frac{q\,\vartheta(\beta)}{(\beta+1)^{q-1}}}}\Big]}{G}&=\frac{(q-1)\beta-\frac{p-q}{q}\,\frac{(\beta+1)^q}{\vartheta(\beta)}}{G}\quad\Rightarrow\nonumber\\
\frac{\Big[{(q-1)\beta-\frac{(p-q)(\beta+1)}{G-p}}\Big]}{G}&=\frac{(q-1)\beta-\frac{p-q}{q}\,\frac{(\beta+1)^q}{\vartheta(\beta)}}{G}\,.
\end{align}
\medskip

\noindent{\bfseries{Claim 2.}} Either of the quantities in (\ref{eq:3p10}) equals to $-\frac{H_p(\beta+1)}{q\,\vartheta(\beta)}$\,.\\
Indeed, we just need to prove that 
\begin{align*}
&\frac{(q-1)\beta-\frac{p-q}{q}\,\frac{(\beta+1)^q}{\vartheta(\beta)}}{G}=-\frac{H_p(\beta+1)}{q\,\vartheta(\beta)}&\Leftrightarrow\\
&\frac{q(q-1)\beta\,\vartheta(\beta)-(p-q)(\beta+1)^q}{q\,\vartheta(\beta)}=-\frac{H_p(\beta+1)\cdot G}{q\,\vartheta(\beta)}&\Leftrightarrow\\
&q(q-1)\big({(p-1)(\beta+1)^q-p(\beta+1)^{q-1}}\big)\,\beta-(p-q)(\beta+1)^q=\\
&{}\hspace{1.0cm}=\big({q(p-1)(\beta+1)-p(q-1)}\big)(-1)\big({q(\beta+1)^{q-1}-(q-1)(\beta+1)^{q}}\big)&\Leftrightarrow\\
&q(q-1)\big({(p-1)(\beta+1)-p}\big)\beta-(p-q)(\beta+1)=\\
&{}\hspace{3.2cm}=-\big({q(p-1)(\beta+1)-p(q-1)}\big)\big[{q-(q-1)(\beta+1)}\big]&\Leftrightarrow\\
&q(q-1)(p-1)(\beta^2+\beta)-pq(q-1)\beta-(p-q)(\beta+1)=\\
&{}\hspace{0.5cm}=-q^2(p-1)(\beta+1)+pq(q-1)+q(p-1)(q-1)(\beta^2+2\beta+1)\,-\\
&{}\hspace{8.5cm}-p(q-1)^2(\beta+1)&\Leftrightarrow\\
&\big[{q(q-1)(p-1)-pq(q-1)-(p-q)}\big]\beta-(p-q)=\\
&{}\hspace{2.cm}=\big[{-q^2(p-1)+2q(p-1)(q-1)-p(q-1)^2}\big]\beta\,+\\
&{}\hspace{2.7cm}+\big({-q^2(p-1)+pq(q-1)+q(p-1)(q-1)-p(q-1)^2}\big)&\Leftrightarrow\\
&\big[{q(pq-q-p+1)-p(q^2-q)-(p-q)}\big]\beta-(p-q)=\\
&{}\hspace{1.cm}=\big[{-pq^2+q^2+2q(pq-p-q+1)-p(q^2-2q+1)}\big]\beta\,+\\
&{}\hspace{1.8cm}+\big({-q^2p+q^2+pq^2-pq+q(pq-p-q+1)-p(q^2-2q+1)}\big)&\Leftrightarrow\\
&\big[{-q^2+2q-p}\big]\beta-(p-q)=\big[{-q^2+2q-p}\big]\beta+(-pq+q-pq+2pq-p)\,,
\end{align*}
which is obviously true. Thus Claim 2 is proved.

By using Claim 2 and (\ref{eq:3p8}), we obtain the following:
\begin{align}\label{eq:3p11}
t^p=s_1+\big({p\,s_1\,\alpha(s_2)}\big)\cdot\bigg({-\frac{H_p(\beta+1)}{q\,\vartheta(\beta)}}\bigg)\,.
\end{align}
But by (\ref{eq:3p6}) we have that:
\begin{align}\label{eq:3p12}
t^{p-q}-\frac{s_1}{s_2}&=-(p-q)\,\frac{s_1\cdot\alpha(s_2)}{q\,\vartheta(\beta)}\qquad\Rightarrow\nonumber\\
\frac{s_1\cdot\alpha(s_2)}{q\,\vartheta(\beta)}&=\frac{1}{p-q}\,\Big({\frac{s_1}{s_2}-t^{p-q}}\Big)\,.
\end{align}
Thus (\ref{eq:3p11}) becomes:
\begin{align}\label{eq:3p13}
	t^p&=s_1-\frac{p}{p-q}\Big({\frac{s_1}{s_2}-t^{p-q}}\Big)\cdot H_q(\beta+1)\qquad\Rightarrow\nonumber\\
	\frac{p-q}{p}\,\frac{t^p-s_1}{t^{p-q}-\frac{s_1}{s_2}}&=H_q(\beta+1)\,.
\end{align}
Define now $\tau(t)=\dfrac{p-q}{p}\,\dfrac{t^p-s_1}{t^{p-q}-\frac{s_1}{s_2}}\,,\; t\geq1\,.$\\
Now by (\ref{eq:3p4}), for $\beta=0$, we have:
\begin{align}\label{eq:3p14}
t^p(0)-\frac{p}{G(0)}\,t^{p-q}(0)&=\frac{p-q}{G(0)}\,s_1-\frac{p}{G(0)}\,\frac{s_1}{s_2}\qquad\Rightarrow\nonumber\\
t^p(0)-\frac{p}{p-q}\,t^{p-q}(0)&=s_1-\frac{p}{p-q}\,\frac{s_1}{s_2}\,.
\end{align}
Define the following function:\; $g(y)=y^p-\frac{p}{p-q}\,y^{p-q}\,,\;\; y\geq1$\,. Then $g$ is strictly increasing on $[1,+\infty)$, with $g(1)=\frac{-q}{p-q}$ and $\lim_{y\to +\infty}g(y)=+\infty$\,.
\medskip

\noindent{\bfseries{Claim 3.}} For every $(s_1,s_2)$ as above we have that:
\[s_1-\frac{p}{p-q}\,\frac{s_1}{s_2}>-\frac{q}{p-q}\,.\]
We prove the above claim as follows: Remember that $(s_1,s_2)$ satisfies:
\[s_1^{q-1}\leq s_2^{p-1}\quad\Rightarrow\quad s_2\geq s_1^{\frac{q-1}{p-1}}\,.\]
Thus for the proof of the claim it is enough to prove
\[s_1-\frac{p}{p-q}\,s_1^{\frac{p-q}{p-1}}>-\frac{q}{p-q}\,,\quad\forall\,s_1\in(0,1]\,.\]
We define $\lambda(s_1):=s_1-\frac{p}{p-q}\,s_1^{\frac{p-q}{p-1}}\,,\; s_1\in(0,1)$\,. Then it is easily seen that:
\begin{align*}
\lambda'(s_1)=1-\frac{p}{p-1}\,s_1^{-\frac{q-1}{p-1}}<1-\frac{p}{p-1}=-\frac{1}{p-1}<0\,.
\end{align*}
\makebox[\linewidth][s]{Thus $\lambda(s_1)$ is strictly decreasing on $(0,1)$, therefore $\lambda(s_1)>\lambda(1)$, or}\\ $s_1-\frac{p}{p-q}\,s_1^{\frac{p-q}{p-1}}>-\frac{q}{p-q}$, which is the desired result. \\
Now we observe that the function $\tau(t)\,,\, t\geq1$, satisfies the following properties:
\begin{enumerate}[label=\roman*)]
\item 
$\tau\big({t(0)}\big)=1$\,.\\
For this claim it is enough to show that  $\dfrac{p-q}{p}\,\dfrac{t^p(0)-s_1}{t^{p-q}(0)-\frac{s_1}{s_2}}=1$\,, which is immediate consequence of (\ref{eq:3p14}).
\item 
$\tau(t)$ is strictly increasing function of $t$ on $[1,+\infty)$. \\
Indeed we calculate 
\begin{align*}
\tau'(t)&\cong p\,t^{p-1}\Big({t^{p-q}-\frac{s_1}{s_2}}\Big)-(t^p-s_1)(p-q)\,t^{p-q-1}\cong\\
 &\cong p\,t^{q}\Big({t^{p-q}-\frac{s_1}{s_2}}\Big)-(p-q)(t^p-s_1)=\\
 &=q\,t^{p}-p\,\frac{s_1}{s_2}\,t^q+(p-q)\,s_1=:\mu(t)\,, \\
 \mu'(t)&=qp\,t^{p-1}-pq\,\frac{s_1}{s_2}\,t^{q-1}\cong t^{p-q}-\frac{s_1}{s_2}\,.
\end{align*}
But for $t>1$ and since  $s_1^{q-1}\leq s_2^{p-1}\;\Rightarrow\; \big({\frac{s_1}{s_2}}\big)^{q-1}\leq s_2^{p-q}\leq 1$\,, we have that $\mu(t)$ is strictly increasing, thus $\mu(t)>\mu(1)$, for every $t>1$, which gives:
\begin{align*}
	\mu(t)&>q-p\,\frac{s_1}{s_2}+(p-q)s_1>q-p\,\frac{s_1}{s_1^{\frac{q-1}{p-1}}}+(p-q)s_1=\\
	&=(p-q)\,s_1+q-p\,s_1^{\frac{p-q}{p-1}}=\Big({s_1-\frac{p}{p-q}\,s_1^{\frac{p-q}{p-1}}}\Big)(p-q)+q\,>\\
	&\hspace{5.2cm}>\Big({-\frac{q}{p-q}}\Big)(p-q)+q=0\,.
\end{align*}
Thus $\mu(t)>0\,,\; \forall\,t>1$\,.\\
From the above we conclude that $\tau'(t)>0\,,\; \forall\,t>1$, thus the desired result.
\end{enumerate}
By  (\ref{eq:3p13}) we have $\tau(t)=H_q(\beta+1)<1$, for every $\beta\in \big({0,\frac{1}{p-1}}\big)$ for which $t'(\beta)=0$, where $t=t(\beta)$. Thus $\beta+1=\omega_q(\tau(t))$, and  since $\tau(t(0))=1$ and $\tau$ strictly increasing, we have from $\tau(t)<1$ that $t<t(0)$.\\ Moreover
\begin{align*}
\vartheta(\beta)&=(p-1)(\beta+1)^q-p\,(\beta+1)^{q-1}=\\
 &=(p-1)\,\omega_q(\tau)^q-p\,\omega_q(\tau)^{q-1}\,.
\end{align*}
Thus by (\ref{eq:3p6}) we obtain 
\begin{align*}
t^{p-q}-\frac{s_1}{s_2}&=\frac{p-q}{q}\,\frac{s_1\,\alpha(s_2)}{p\,\omega_q(\tau)^{q-1}-(p-1)\,\omega_q(\tau)^q}\qquad\Rightarrow\\
F_{s_1,s_2}(t)&:=q\,\big({p\,\omega_q(\tau)^{q-1}-(p-1)\,\omega_q(\tau)^q}\big)\Big({t^{p-q}-\frac{s_1}{s_2}}\Big)-(p-q)\,s_1\,\alpha(s_2)=0\,.
\end{align*}
We have proved the following:
\begin{theorem}\label{thm:3p1}
Let $s_1,s_2$ fixed such that $0<s_1^{q-1}\leq s_2^{p-1}\leq1$. If there exists $\beta\in \big({0,\frac{1}{p-1}}\big)$ such that $t'(\beta)=0$, then $t=t(\beta)$ should satisfy $F_{s_1,s_2}(t)=0$, where $F_{s_1,s_2}(t)$ as given right above, and $\tau(t)$ is described previously. Moreover in this case $t(0)>1$ and $t=t(\beta)\in\big[{1,t(0)}\big)$\,.
\end{theorem}
Note that $A_{\beta}=\frac{G(\beta)}{p(\beta+1)^{q}}=\mc{K}(\beta)>\mc{K}(0)=\frac{p-q}{p}\,,\;\forall\,\beta\in\big(0,\tfrac{1}{p-1}\big]$\,. Thus $A_{\frac{1}{p-1}}>\frac{p-q}{p}$\,, thus the function $h_{\frac{1}{p-1}}:[1,+\infty)\longrightarrow\big({-\infty,h_{\frac{1}{p-1}}(1)}\big]$  defined by $h_{\frac{1}{p-1}}(y)=y^{p-q}-A_{\frac{1}{p-1}}y^p$ is strictly decreasing on $[1,+\infty)$.\\
Thus the quantity $t\big({\frac{1}{p-1}}\big)$ is well defined, from (\ref{eq:3p4}) for $\beta=\frac{1}{p-1}$. Now define $t'\big({\frac{1}{p-1}}\big)=\lim_{\begin{smallmatrix}
\beta\to \frac{1}{p-1}\\
0<\beta<\frac{1}{p-1}
\end{smallmatrix}}t'(\beta)$. This limit is well defined for the following reasons: Remember that $t'(\beta)$ satisfies (\ref{eq:3p5}), for every $\beta\in \big({0,\frac{1}{p-1}}\big)$, that is 
\begin{align}\label{eq:3p15}
	&p\,t^{p-q-1}(\beta)\, \Big[{t^{q}(\beta)-(p-q)\,\frac{(\beta+1)^q}{G(\beta)}}\Big]\, t'(\beta)-pt^{p-q}(\beta)q(q-1)\frac{\vartheta(\beta)}{G^2(\beta)}=\nonumber\\
	&=(p-q)s_1\frac{-p(q-1)}{G^2(\beta)}+\frac{s_1}{s_2}\bigg[{p(q-1)\frac{p-q}{G^2(\beta)}\,\omega_q(s_2)^q -\frac{p\,q(q-1)\vartheta(\beta)}{G^2(\beta)}}\bigg],
\end{align}
where $G=G(\beta)=q(p-1)(\beta+1)-p(q-1)$\,.\\
Letting $\beta\to \tfrac{1}{p-1}$ in (\ref{eq:3p15}), since $G(\beta)\to qp-pq+p=p$ we  get that the right side tends to \[(p-q)s_1\frac{-p(q-1)}{p^2}+\frac{s_1}{s_2}\,p(q-1)\frac{p-q}{p^2}\,\omega_q(s_2)^q \in \mb{R}\,.\]
As we have noted above $t\big({\frac{1}{p-1}}\big)$ is well defined, and by its definition equals to $\lim_{\beta\to {\frac{1}{p-1}}^{-}}t(\beta)$. Thus by (\ref{eq:3p15}) we have as a consequence that the limit $\lim_{\beta\to {\frac{1}{p-1}}^{-}}t'(\beta)$ exists if we prove that $t^q\,(\frac{1}{p-1})-(p-q)\frac{\big({\frac{p}{p-1}}\big)^q}{p}\neq0$\,. Indeed this is the case, since
\begin{align}\label{eq:3p16}
\frac{p-q}{p}\,\Big({\frac{p}{p-1}}\Big)^q<1\leq t\,\Big({\frac{1}{p-1}}\Big)^q\,.
\end{align}
The second inequality in (\ref{eq:3p16}) is obviously true, while for the first we argue as follows:\\
For $1<q<p$ consider the $L^p$-integrable function on $(0,1]$, defined by $\delta(t)=t^{-1/p}$\,. Then since $\delta(t)$ is non-constant on $(0,1]$ we set (by H\"{o}lder's inequality):
\begin{align*}
\bigg({\int_{0}^{1}t^{-1/p}\,dt}\bigg)^q<\int_{0}^{1}t^{-q/p}\,dt\quad&\Rightarrow\quad\Big({\frac{p}{p-1}}\Big)^q<\frac{1}{1-\frac{q}{p}}=\frac{p}{p-q}\\
&\Rightarrow\quad\Big({\frac{p}{p-1}}\Big)^q\,\frac{p-q}{p}<1\,.
\end{align*}
Thus taking limits in (\ref{eq:3p5}) as $\beta \to {\frac{1}{p-1}}^-$ we have:
\begin{align}\label{eq:3p17}
	&p\,t^{p-1}\big({\tfrac{1}{p-1}}\big)\, t'\big({\tfrac{1}{p-1}}\big)-p(p-q)\,t^{p-q-1}\big({\tfrac{1}{p-1}}\big)\,\frac{\big({\tfrac{p}{p-1}}\big)^q}{p}\, t'\big({\tfrac{1}{p-1}}\big)=\nonumber\\
	&\hspace{2.0cm}=\frac{-p(p-q)(q-1)}{p^2}\,s_1+\frac{s_1}{s_2}\,\frac{p(q-1)(p-q)}{p^2}\,\omega_q(s_2)^q &\Rightarrow\nonumber\\
&t'\big({\tfrac{1}{p-1}}\big)\,t^{p-q-1}\big({\tfrac{1}{p-1}}\big)\,\bigg[{t^{q}\big({\tfrac{1}{p-1}}\big)-\frac{p-q}{p}\,\Big({\frac{p}{p-1}}\Big)^q}\bigg]=\nonumber\\
&\hspace{1.0cm}=\frac{-(p-q)(q-1)}{p^2}\,s_1+\frac{s_1}{s_2}\,\frac{(q-1)(p-q)}{p^2}\,\omega_q(s_2)^q=\nonumber\\
&\hspace{1.5cm}=\frac{(p-q)(q-1)}{p^2}\,s_1\Big({\frac{\omega_q(s_2)^q}{s_2}-1}\Big)=\frac{(p-q)(q-1)}{p^2}\,s_1\alpha(s_2)\,.
\end{align}
At this point we note that we restrict to ourselves (see the beginning of Section \ref{sec:3}) to the study of those functions $h$ which are not constant on $(0,\kappa]$. Thus $x,y$ should satisfy:
$\frac{x^q}{\kappa^{q-1}}<y\;\Rightarrow\; s_2=\frac{x^q}{\kappa^{q-1}y}<1$\,.\\
It is obvious that the function $\alpha(s_2)=\big({\frac{\omega_q(s_2)^q}{s_2}-1}\big)$ is strictly decreasing on $s_2 \in (0,1]$ so, $\alpha(s_2)>\alpha(1)=0$. Thus by (\ref{eq:3p17}) and (\ref{eq:3p16}) we  get that $t'\big({\tfrac{1}{p-1}}\big)>0$\,.\\
Again now by (\ref{eq:3p5}), or (\ref{eq:3p15}) we get that $t'(0):=\lim_{\beta\to0^{+}}t'(\beta)$ exists if $t^q(0)\neq (p-q)\,\frac{1}{G(0)}=1$\,. But $t(0)$ satisfies (by its definition) : $\delta^p=\frac{\int_{0}^{\kappa}\big(\frac{1}{t}\int_{0}^{t}h\big)^p dt}{\int_{0}^{\kappa}h^p}\leq t^p(0)$, and since $h$ is not constant we easily see that $\delta>1$, which gives $t(0)>1$. Thus $t'(0):=\lim_{\beta\to0^{+}}t'(\beta)$ is well defined. \\
Now we consider the following general cases:
\begin{enumerate}[label=\Roman*)]
\item 
$t'(0)<0$. Then since $t'\big({\tfrac{1}{p-1}}\big)>0\;\Rightarrow\; \exists\,\beta\in \big({0,\frac{1}{p-1}}\big):\; t'(\beta)=0$, for which if we set $t=t(\beta)$, $t$ satisfies $F_{s_1,s_2}(t)=0$. But as can be easily seen the function $F_{s_1,s_2}(\cdot)$ is strictly increasing in $[1,+\infty)$, that is $F$ is one to one. Thus $t(\beta)$ is uniquely determined when $t'(\beta)=0$. \\
This easily gives in this case that for a $\beta\in \big({0,\frac{1}{p-1}}\big)$ such that $t'(\beta)=0$ we have that 
\[t=t(\beta)=\min\Big\{{t(\gamma)\;:\;\gamma\in\big[0,\tfrac{1}{p-1}\big]}\Big\}\]
and of course $F_{s_1,s_2}(t)=0$. This case is then completely studied. We turn now to the other cases.
\item 
$t'(0)>0$. Then we must have that $t'(\beta)>0$,\; $\forall\,\beta\in \big({0,\frac{1}{p-1}}\big)$, otherwise there would exist $\beta\in \big({0,\frac{1}{p-1}}\big)$ for which $t(\beta)>t(0)$ with $t'(\beta)=0$. But the condition $t'(\beta)=0$ implies $t(\beta)<t(0)$ as we have seen in the proof of Theorem \ref{thm:3p1}. Thus in this case 
\[\min\Big\{{t(\gamma)\;:\;\gamma\in\big[0,\tfrac{1}{p-1}\big]}\Big\}=t(0)\,.\]
\item 
$t'(0)=0$. There does not exist $\beta\in \big({0,\frac{1}{p-1}}\big)$ for which $t'(\beta)=0$ and $t(\beta)<t(0)$, by using the same arguments that are mentioned above. Thus  $t(\beta)\geq t(0)$, \; $\forall\,\beta\in \big({0,\frac{1}{p-1}}\big)$, since otherwise would exist $\beta\in \big({0,\frac{1}{p-1}}\big)$ with $t(\beta)<t(0)$ that is $t(\gamma)$ would be strictly decreasing on a subinterval of  $\big({0,\frac{1}{p-1}}\big)$. This fact, and the observation that $t'\big({\frac{1}{p-1}}\big)>0$ would then give the existence of an $\beta_1\in\big({0,\frac{1}{p-1}}\big)$ for which $t'(\beta_1)=0$ and $t(\beta_1)<t(0)$ which is a case that is excluded. Thus $t(\beta)\geq t(0)$, \; $\forall\,\beta\in \big({0,\frac{1}{p-1}}\big)$ and we get in this case also that
\[\min\Big\{{t(\gamma)\;:\;\gamma\in\big[0,\tfrac{1}{p-1}\big]}\Big\}=t(0)\,.\]
Moreover we can argue alternatively for the third case as follows: If $t'(0)=0$, then $F_{s_1,s_2}(t(0))=0$. Thus we must have that $t'(\beta)>0$, \; $\forall\,\beta\in \big({0,\frac{1}{p-1}}\big)$.\\
Indeed if there exists $\beta\in \big({0,\frac{1}{p-1}}\big)$ with $t'(\beta)\leq 0$, then since\\
$t'\big({\tfrac{1}{p-1}}\big)>0\;\Rightarrow\; \exists\,\beta_1\in \big[{\beta,\frac{1}{p-1}}\big):\; t'(\beta_1)=0\;\Rightarrow\;F_{s_1,s_2}(t(\beta_1))=0$ and $t(\beta_1)<t(0)$, as already proved. Thus $F_{s_1,s_2}(t(0))=0=F_{s_1,s_2}(t(\beta_1))$ with $t(\beta_1)<t(0)$ which can not be true. Thus $t'(\beta)>0$, \; $\forall\,\beta\in \big({0,\frac{1}{p-1}}\big)$ and as consequence 
\[\min\Big\{{t(\gamma)\;:\;\gamma\in\big[0,\tfrac{1}{p-1}\big]}\Big\}=t(0)\,,\]
in case where $t'(0)=0$.
\end{enumerate}
Note that in case where $t'(0)=0$ we get by following the lines of the proof of Theorem \ref{thm:3p1}, that:
\[\tau(t(0))=H_q(1)=1\,.\]
Thus $F_{s_1,s_2}(t(0))=0$ (see right above the statement of  Theorem \ref{thm:3p1}).\\
Note also that in case where $t'(0)>0\;\Rightarrow\; F_{s_1,s_2}(t(0))<0$.\\
The above claim is true and for this we argue as follows:\\
If $t'(0)>0$, then by (\ref{eq:3p5}) or (\ref{eq:3p15}) we have 
\begin{align}\label{eq:3p18}
	&p\,t^{p-q-1}(0)\, t'(0)\, \big[{t^{q}(0)-1}\big]=\nonumber\\
	&\qquad=p\,t^{p-q}(0)\,q(q-1)\frac{-1}{(p-q)^2}+(p-q)\,s_1\frac{-p(q-1)}{(p-q)^2}\,+\nonumber\\
	&\hspace{3.0cm}+\frac{s_1}{s_2}\bigg[{p(q-1)\frac{p-q}{(p-q)^2}\,\omega_q(s_2)^q -\frac{p\,q(q-1)(-1)}{(p-q)^2}}\bigg]\,.
\end{align}
But since $t(0)$ bounds the quantity $\delta=\big({\frac{I}{z}}\big)^{1/p}=\Big({\frac{\int_{0}^{\kappa}\big(\frac{1}{t}\int_{0}^{t}h\big)^p dt}{\int_{0}^{\kappa}h^p}}\Big)^{1/p}$ and $h$ is considered to be non-constant we have $t(0)>1$, and since also $t'(0)>0$, we obtain by (\ref{eq:3p18})that its right hand side is positive, that is 
\begin{align}\label{eq:3p19}
&-q\,t^{p-q}(0)-(p-q)\,s_1+\frac{s_1}{s_2}\big[{(p-q)\,\omega_q(s_2)^q+q}\big]>0&\Rightarrow\nonumber\\
&q\,t^{p-q}(0)<(p-q)\,\frac{s_1}{s_2}\,\omega_q(s_2)^q-(p-q)\,s_1+ q\,\frac{s_1}{s_2}&\Rightarrow\nonumber\\
&q\,\Big({t^{p-q}(0)-\frac{s_1}{s_2}}\Big)<(p-q)\,s_1\cdot\alpha(s_2)\,.
\end{align}
But then by definition of  $F_{s_1,s_2}(t(0))=$\\
$q\,\big({p\,\omega_q(\tau(t(0)))^{q-1}-(p-1)\,\omega_q(\tau(t(0)))^q}\big)\Big({t(0)^{p-q}-\frac{s_1}{s_2}}\Big)-(p-q)\,s_1\,\alpha(s_2)$
and since $\tau(t(0))=1$, and (\ref{eq:3p19}) is true, we get $F_{s_1,s_2}(t(0))<0$. [In the same way we prove that: if  $t'(0)=0$, then $F_{s_1,s_2}(t(0))=0$].

At this point we mention that accordingly to the results of this section, we have that whenever $0<\frac{x^q}{\kappa^{q-1}}<y\leq x^{\frac{p-q}{p-1}}\cdot z^{\frac{q-1}{p-1}}\;\Leftrightarrow\; 0<s_1^{\frac{q-1}{p-1}}\leq s_2<1$, there is defined a constant $t=t(s_1,s_2)$ for which if $h:(0,\kappa]\longrightarrow\mb{R}^{+}$ satisfies  $\int_{0}^{\kappa}h=x$\,,\; $\int_{0}^{\kappa}h^q=y$ and $\int_{0}^{\kappa}h^p=z$ then
\begin{align}\label{eq:3p20}
	\int_{0}^{\kappa}\bigg(\frac{1}{t}\int_{0}^{t}h\bigg)^p dt\leq t^p(s_1,s_2)\cdot\int_{0}^{\kappa}h^p
\end{align}
where $t(s_1,s_2)=t$ is the greatest element of $\big[{1,t(0)}\big]$ for which $F_{s_1,s_2}(t)\leq 0$. Moreover for each such fixed $s_1,s_2$ 
\[t=t(s_1,s_2)=\min\Big\{{t(\beta)\;:\;\beta\in\big[0,\tfrac{1}{p-1}\big]}\Big\}\]
where $t(\beta)=t(\beta,s_1,s_2)$ is defined by equation (\ref{eq:3p4}). That is we have found a constant $t=t(s_1,s_2)$ for which inequality (\ref{eq:3p20}) is satisfied for all $h:(0,\kappa]\longrightarrow\mb{R}^{+}$ as mentioned above. Note that $s_1,s_2$ depend by a certain way on $x,y,z$, namely $s_1=\frac{x^p}{\kappa^{p-1}z}$,  $s_2=\frac{x^q}{\kappa^{q-1}y}$ and $F_{s_1,s_2}(\cdot)$ is given in terms of $s_1,s_2$.

\vspace{50pt}
\noindent Nikolidakis Eleftherios\\
Assistant Professor\\
Department of Mathematics \\
Panepistimioupolis, University of Ioannina, 45110\\
Greece\\
E-mail address: enikolid@uoi.gr

\end{document}